\documentclass[review]{elsarticle}

\usepackage{lineno,hyperref}

\usepackage[english]{babel}
\usepackage[utf8]{inputenc}
\usepackage[T1]{fontenc}
\usepackage{graphics}
\usepackage{booktabs}
\usepackage{amsmath}
\usepackage{inputenc}
\usepackage{amsthm}
\usepackage{amssymb}
\usepackage{setspace}
\usepackage{xcolor}
\usepackage{array}
\usepackage{natbib}
\usepackage{graphicx}
\usepackage{longtable}
\usepackage{dsfont}
\usepackage{mathtools}
\usepackage{subfig}
\usepackage{bbm}
\usepackage{rotating}
\usepackage{url}
\usepackage{varioref}

\modulolinenumbers[5]

\newtheorem{remark}{Remark}

\newtheorem{proposition}{Proposition}
\newtheorem{lemma}{Lemma}

\DeclareMathOperator\erf{erf}
\DeclareMathOperator\erfc{erfc}

\begin{document}

\begin{frontmatter}

\title{A fractional Hawkes process II: \\ Further characterization of the process\tnoteref{mytitlenote}}
%\tnotetext[mytitlenote]{Fractional Hawkes Process}

\author[mysecondaryaddress,myfifthaddress]{Cassien Habyarimana}
\ead{habyarimanacassien@gmail.com}

\author[myaddress]{Jane A. Aduda}
\ead{jaduda@jkuat.ac.ke}

\author[myfourthaddress,mysapienzaaddress]{Enrico Scalas\corref{mycorrespondingauthor}}
\ead{e.scalas@sussex.ac.uk}

\author[mymainaddress]{Jing Chen}
\ead{chenj60@cardiff.ac.uk}

\author[mythirdaddress]{Alan G. Hawkes}
\ead{a.g.hawkes@swansea.ac.uk}

\author[mytorinoaddress]{Federico Polito}
\ead{federico.polito@unito.it}

\address[mysecondaryaddress]{Department of Mathematics, Pan African University Institute of Sciences, Technology and Innovation, Nairobi, Kenya}
\address[myfifthaddress]{Department of Mathematics, University of Rwanda, Kigali, Rwanda}
\address[myaddress]{Department of Statistics and Actuarial Sciences, Jomo Kenyatta University of Agriculture and Technology, Nairobi, Kenya}
\address[myfourthaddress]{Department of Mathematics, University of Sussex, Falmer, Brighton, UK}
\address[mysapienzaaddress]{Department of Statistical Sciences, Sapienza University of Rome, Rome, Italy}
\address[mymainaddress]{School of Mathematics, Cardiff University, Cardiff, UK}
\address[mythirdaddress]{School of Management, Swansea University, Swansea, UK}
\address[mytorinoaddress]{Department of Mathematics, University of Turin, Turin, Italy}

\begin{abstract}
We characterize a Hawkes point process with kernel proportional to the probability density function of Mittag-Leffler random variables. This kernel decays as a power law with exponent $\beta +1 \in (1,2]$. Several analytical results can be proved, in particular for the expected intensity of the point process and for the expected number of events of the counting process. These analytical results are used to validate algorithms that numerically invert the Laplace transform of the expected intensity as well as Monte Carlo simulations of the process. Finally, Monte Carlo simulations are used to derive the full distribution of the number of events. The algorithms used for this paper are available at {\tt https://github.com/habyarimanacassien/Fractional-Hawkes}.
\end{abstract}


\end{frontmatter}

%\linenumbers

\section{Introduction}
\label{sec:Introduction}
In reference \cite{chen20}, a Hawkes process of fractional type was introduced. A Hawkes process \cite{hawkes71a,hawkes71b} is a self-exciting point process defined by the following conditional intensity
\begin{equation}
\label{condint}
\Lambda(t|\mathcal{H}_t) = \Lambda_0 + \alpha \int_0^t f(t-u) \, dN(u),
\end{equation}
where $\mathcal{H}_t$ represents the history of the process, $\Lambda_0 > 0$ is a ``naked'' rate (if $\alpha =0$ we have a Poisson process with rate $\Lambda_0$), $\alpha \in (0,1)$ is the branching ratio (see \cite{hawkes74} for the connection with branching processes), and $f(t)$ is the probability density function of a positive random variable.
The idea was to use as kernel the probability density function of Mittag-Leffler random variables defined as follows \cite{mainardigorenflo}
\begin{equation}
\label{mlrv}
f_\beta (t) = \int_0^\infty \theta \mathrm{e}^{-\theta t} K_\beta (\theta) \; d \theta,
\end{equation}
where $\beta \in (0,1]$ and $K_\beta (\theta)$ is given by
\begin{equation}
\label{mlkernel}
K_\beta (\theta) = \frac{1}{\pi} \frac{\theta^{\beta-1} \sin (\beta \pi)}{\theta^{2 \beta} + 2 \theta^\beta \cos (\beta \pi) +1}.
\end{equation}
Equation \eqref{mlrv} defines the probability density function of a Mittag-Leffler random variable interpolating between a stretched exponential for small $t$ and a power law of index $\beta$ for large $t$ \cite{mainardi00, mainardi04}. The behaviour of the Mittag-Leffler distribution is similar to the one of the Pareto distribution \cite{georgiou15}, but it is easy to use it as there is a simple analytic formula for the Laplace transform of $f_\beta (t)$
\cite{mainardigorenflo}. Functions such as \eqref{mlrv} play an important role in fractional calculus \cite{mainardi00} and, for this reason, we decided to call the corresponding Hawkes process a ``fractional'' Hawkes process. 

\noindent \textcolor{black}{Expressing the measure $d N(u)$ in equation \eqref{condint} as $d N(u) = \sum_{T_k \leq t} \delta(u-T_k)$ where $\{T_k\}_{k=1}^\infty$ are the random variables representing the event times/epochs, one gets
$$
\Lambda(t|\mathcal{H}_t) = \Lambda_0 + \alpha \sum_{T_k \leq t} f_\beta (t - T_k).
$$
The probability density function $f_\beta (t)$ diverges in $0^+$:
$$
\lim_{t \to 0^+} f_\beta (t) = \infty,
$$
so, one potentially has infinite activity in all the epochs $T_k$s. However, one can avoid this problem by setting
$$
\Lambda(t|\mathcal{H}_t) = \Lambda_0 + \alpha \sum_{T_k < t} f_\beta (t - T_k), \; \mathrm{for} \; t \neq T_k
$$
and
$$
\Lambda(t|\mathcal{H}_t) = \Lambda(T_k^-|\mathcal{H}_t), \; \mathrm{for} \; t = T_k.
$$
In this way $\Lambda(T_k|\mathcal{H}_t)$ is finite for all the epochs and, for $\varepsilon > 0$ and $\varepsilon \ll 1$, one has that $\Lambda(T_k + \varepsilon|\mathcal{H}_t)$ is finite as well.}

Our simple approach is not the unique way for defining a ``fractional'' Hawkes process. For example, Hainaut \cite{hainaut19} considers a time-changed intensity process $\Lambda(S_t)$ where the conditional intensity of the self-exciting process is  solution of a stochastic differential equation
\begin{equation}
\label{sde}
d \Lambda_t = \kappa (\rho - \Lambda_t) dt + \eta d P_t,
\end{equation}
where $\kappa$, $\rho$ and $\eta$ are suitable parameters and the driving process $P_t$ is a continuous-time random walk defined as
\begin{equation}
\label{ctrw}
P_t = \sum_{i=1}^{N_t} \xi_i
\end{equation}
where $N_t$ is a counting process and $\xi_i$s are independent and identically distributed marks with finite positive
mean and finite variance. The time-change $S_t$ is then the inverse of a $\beta$-stable subordinator. This process has been later applied to CDS valuation \cite{hainaut22}. Our process is related to {\em Epidemic
Type Aftershock Sequence} (ETAS) models introduced by Ogata \cite{ogata88} as the kernel in \eqref{mlrv} has the same asymptotic behaviour as the power-law kernel suggested by Ogata. 

In the previous paper \cite{chen20}, we used the Laplace transform of the probability density function \eqref{mlrv} \cite{mainardi00}
\begin{equation}
\label{ltmlrv}
\tilde{f}_\beta (s) = \int_0^\infty \mathrm{e}^{-st} f_\beta (t) \; dt = \frac{1}{1+s^\beta}
\end{equation}
to derive the Laplace transform of the expected intensity 
\begin{equation}
\label{eint}
\lambda (t) = \mathbb{E} [\Lambda(t|\mathcal{H}_t)].
\end{equation}
Indeed, from the definition of the conditional intensity process in equation \eqref{condint}, one can derive a self-consistent equation for $\lambda(t)$
\begin{equation}
\label{selfconsistent}
\lambda(t) = \Lambda_0 + \int_0^t f_\beta (t -u) \lambda(u) \; du,
\end{equation}
and, taking the Laplace transform of this equation, one eventually gets
\begin{equation}
\label{expectedintensity1}		  
\tilde{\lambda}(s)=\frac{\Lambda_0}{s} \frac{1+s^\beta}{(1-\alpha)+s^\beta}.
\end{equation}
If time is rescaled using a time scale $\gamma > 0$, this equation becomes
\begin{equation}
\label{expectedintensity}		  
\tilde{\lambda}(s)=\frac{\Lambda_0}{s} \frac{\gamma+s^\beta}{(1-\alpha)\gamma+s^\beta}.
\end{equation}
In \cite{chen20}, we noticed that there is an analytical formula for the expected intensity $\lambda (t)$ when $\beta = 1/2$, but no proof was given of this result. This is now discussed in detail in section \ref{sec:EI}. In that same section, we now present \textcolor{black}{(i) the analytical inverse of \eqref{expectedintensity}} as well as (ii) a numerical method to numerically invert the Laplace transform and compute $\lambda(t)$, in the general case $\beta \in (0,1]$. Note that the case $\beta=1$ leads to probability density functions for exponential random variables. Finally, we study the asymptotic behaviour of $\lambda(t)$ as a consequence of Tauberian theorems, a point that was also mentioned in \cite{chen20} without explicit proof. In Section \ref{sec:ENE}, results on the expected number of events up to time $t$ are presented whereas the distribution of the number of events is discussed in Section \ref{sec:dist}. The Monte Carlo thinning algorithm \cite{ogata81} described in \cite{chen20} and generating the intensity process has been translated into R \cite{Rproject} and it is used throughout this paper. The code used in the present paper is freely available from the following repository {\tt https://github.com/habyarimanacassien/Fractional-Hawkes}.

\section{Expected intensity}
\label{sec:EI}

In this section, we present our results on the expected intensity as a function of time. We focus on the analytical inversion of its Laplace transform in the case $\beta =1/2$, \textcolor{black}{and in the general case $\beta \in (0,1)$ as well as} on the numerical inversion, and on the limiting behaviour for large times ($t \to \infty$). 

\subsection{Analytical result for $\beta = 1/2$}

As discussed in the previous section, the Laplace transform of the expected intensity is given by equation \eqref{expectedintensity}.
As mentioned above, in \citep{chen20}, we presented the explicit inverse Laplace transform of this equation in the case $\beta = 1/2$ without proof. We are now presenting a proof in the following proposition.
\begin{proposition}
The inverse Laplace transform of (\ref{expectedintensity})  for $\beta=1/2$ is given by
\begin{equation}
\lambda(t)=\mathcal{L}^{-1}\left\{\tilde{\lambda} (s) \right\}(t)=\frac{\Lambda_0}{1-\alpha}-\frac{\alpha\Lambda_0}{1-\alpha}e^{(1-\alpha)^2\gamma^2t}\erfc{\left[(1-\alpha)\gamma\sqrt{t}\right]},
\label{inv1}
\end{equation}
where $$\erfc(t)=1-\erf(t)=\frac{2}{\sqrt{\pi}}\int^\infty_te^{-\tau^2}\,d\tau$$ is the complementary of the error function 
$$\erf(t)=\frac{2}{\sqrt{\pi}}\int^t_0e^{-\tau^2}\,d\tau.$$
\label{proposition1}
\end{proposition}
\begin{proof}
The strategy is to find the Laplace transform of $\lambda(t)$ in \eqref{inv1} and show that it coincides with $\tilde{\lambda}(s)$ in \eqref{expectedintensity}. We have
\begin{align}
&\mathcal{L}\left\{\frac{\Lambda_0}{1-\alpha}-\frac{\alpha\Lambda_0}{1-\alpha}e^{(1-\alpha)^2\gamma^2t}\erfc{\left[(1-\alpha)\gamma\sqrt{t}\right]}\right\} (s)\nonumber\\
&=\mathcal{L}\left\{\frac{\Lambda_0}{1-\alpha}\right\}(s)-\mathcal{L}\left\{\frac{\alpha\Lambda_0}{1-\alpha}e^{(1-\alpha)^2\gamma^2t}\right\} (s)+ \nonumber \\ 
&\mathcal{L}\left\{\frac{\alpha\Lambda_0}{1-\alpha}e^{(1-\alpha)^2\gamma^2t}\erf{\left[(1-\alpha)\gamma\sqrt{t}\right]}\right\} (s) \nonumber\\
&=\frac{\Lambda_0}{1-\alpha}\cdot\frac{1}{s}-\frac{\alpha\Lambda_0}{1-\alpha}\cdot\frac{1}{s-(1-\alpha)^2\gamma^2}+\frac{\alpha\Lambda_0}{1-\alpha}\mathcal{L}\left\{e^{(1-\alpha)^2\gamma^2t}\erf{\left[(1-\alpha)\gamma\sqrt{t}\right]}\right\} (s).
\label{lap1}
\end{align}
Using the property that $\mathcal{L}\left\{e^{at}f(t)\right\} (s) =\tilde{f}(s-a)$ where
$\tilde{f}(s)=\mathcal{L}\left\{f(t)\right\} (s)$ \citep{oberhettinger2012tables}, the expression in \eqref{lap1} becomes
\begin{multline}
\mathcal{L}\left\{\frac{\Lambda_0}{1-\alpha}-\frac{\alpha\Lambda_0}{1-\alpha}e^{(1-\alpha)^2\gamma^2t}\erfc{\left[(1-\alpha)\gamma\sqrt{t}\right]}\right\} (s)\\ 
=\frac{\Lambda_0}{1-\alpha} \frac{1}{s}-\frac{\alpha\Lambda_0}{1-\alpha} \frac{1}{s-(1-\alpha)^2\gamma^2}+\frac{\alpha\Lambda_0}{1-\alpha}F\left(s-(1-\alpha)^2\gamma^2\right).
\label{lap2}
\end{multline}
$F(s)$ in \eqref{lap2} is given by
\begin{equation}
F(s)=\mathcal{L}\left\{\erf\sqrt{(1-\alpha)^2\gamma^2t}\right\}(s).
\label{lap3}
\end{equation}			
Also, using the property that $\mathcal{L}\left\{f(at)\right\} (s) =\frac{1}{a} \tilde{f} (\frac{s}{a})$ where, again, $\tilde{f}(s)=\mathcal{L}\left\{f(t)\right\} (s)$ \citep{oberhettinger2012tables}, the expression in \eqref{lap3} becomes
\begin{equation}
F(s)=\mathcal{L}\left\{\erf\sqrt{(1-\alpha)^2\gamma^2t}\right\} (s) =\frac{1}{(1-\alpha)^2\gamma^2}\cdot G\left(\frac{s}{(1-\alpha)^2\gamma^2}\right),
\end{equation}
where 
$$G(s)=\mathcal{L}\left\{\erf\sqrt{t}\right\} (s).$$
We are now left with the task of computing  $G(s)=\mathcal{L}\left\{\erf{\sqrt{t}}\right\} (s)$.
With $\erf(\sqrt{t})=\frac{2}{\sqrt{\pi}}\int^{\sqrt{t}}_0e^{-\tau^2}\,d\tau$, let $\tau^2=u$, so that $d\tau=\frac{du}{2\sqrt{u}}$, and $u=0$ for $\tau=0$ and $u=\sqrt{t}$ for $\tau=t$. Thus, we have
\begin{equation*}
\erf(\sqrt{t})=\frac{2}{\sqrt{\pi}}\int^t_0e^{-u}\,\frac{du}{2\sqrt{u}}=\frac{1}{\sqrt{\pi}}\int^t_0e^{-u}u^{-1/2}\,du,
\end{equation*}
which implies that
\begin{equation}
G(s)=\mathcal{L}\left\{\erf(\sqrt{t})\right\}(s)=\frac{1}{\sqrt{\pi}}\mathcal{L}\left\{\int^t_0e^{-u}u^{-1/2}\,du\right\}(s).
\label{lap4}
\end{equation}
Using the property that
$$
\mathcal{L}\left\{\int^t_0f(\tau)\,d\tau\right\} (s) =\frac{1}{s} \tilde{f}(s),
$$
where, once more,
$\tilde{f}(s)=\mathcal{L}\left\{f(t)\right\} (s)$ \citep{oberhettinger2012tables}, the expression in \eqref{lap4} becomes
\begin{equation}
G(s)=\mathcal{L}\left\{\erf(\sqrt{t})\right\}(s)=\frac{1}{\sqrt{\pi}}\frac{\mathcal{L}\left\{e^{-u}u^{-1/2}\right\}(s)}{s}=\frac{1}{\sqrt{\pi}}\frac{H(s+1)}{s},
\label{lap5}
\end{equation}
where
\begin{equation}
H(s)=\mathcal{L}\{u^{-1/2}\} (s) =\frac{\Gamma(-\frac{1}{2}+1)}{s^{-\frac{1}{2}+1}}(s)=\frac{\Gamma(\frac{1}{2})}{s^{\frac{1}{2}}}=\frac{\sqrt{\pi}}{\sqrt{s}}.
\label{h-value}
\end{equation}
Therefore, replacing \eqref{h-value} into \eqref{lap5}, yields
\begin{align}
G(s)&=\frac{1}{\sqrt{\pi}}\frac{H(s+1)}{s}=\frac{1}{\sqrt{\pi}}\frac{\frac{\sqrt{\pi}}{\sqrt{s+1}}}{s}=\frac{1}{s\sqrt{s+1}}.\nonumber\\
&\Longrightarrow	F(s)=\frac{1}{(1-\alpha)^2\gamma^2}\cdot G\left(\frac{s}{(1-\alpha)^2\gamma^2}\right)=\frac{(1-\alpha)\gamma}{s\sqrt{s+(1-\alpha)^2\gamma^2}}\nonumber\\
&\Longrightarrow	F\left(s-(1-\alpha)^2\gamma^2\right)=\frac{(1-\alpha)\gamma}{\left(s-(1-\alpha)^2\gamma^2\right)\sqrt{s-(1-\alpha)^2\gamma^2+(1-\alpha)^2\gamma^2}}\nonumber\\
&\qquad\qquad\qquad\qquad\qquad =\frac{(1-\alpha)\gamma}{\left(s-(1-\alpha)^2\gamma^2\right)\sqrt{s}}.
\label{f-value}
\end{align}
Finally, taking the expression in (\ref{f-value}) into (\ref{lap2}), yields
\begin{align*}
&\mathcal{L}\left\{	\frac{\Lambda_0}{1-\alpha}-\frac{\alpha\Lambda_0}{1-\alpha}e^{(1-\alpha)^2\gamma^2t}\erfc{\left[(1-\alpha)\gamma\sqrt{t}\right]}\right\} (s)\nonumber\\
&=\frac{\Lambda_0}{1-\alpha} \frac{1}{s}-\frac{\alpha\Lambda_0}{1-\alpha} \frac{1}{s-(1-\alpha)^2\gamma^2}+\frac{\alpha\Lambda_0}{1-\alpha}F\left(s-(1-\alpha)^2\gamma^2\right)
\nonumber\\
&=\frac{\Lambda_0}{1-\alpha} \frac{1}{s}-\frac{\alpha\Lambda_0}{1-\alpha} \frac{1}{s-(1-\alpha)^2\gamma^2}+\frac{\alpha\Lambda_0}{1-\alpha}\frac{(1-\alpha)\gamma}{\left(s-(1-\alpha)^2\gamma^2\right)\sqrt{s}}\nonumber\\
&=\frac{\Lambda_0}{1-\alpha}\left[\frac{\left(\sqrt{s}-(1-\alpha)\gamma\right)\left(\sqrt{s}+(1-\alpha)\gamma-\alpha\sqrt{s}\right)}{s\left(\sqrt{s}-(1-\alpha)\gamma\right)\left(\sqrt{s}+(1-\alpha)\gamma\right)}\right]\nonumber\\
&=\frac{\Lambda_0}{s} \frac{\gamma+\sqrt{s}}{(1-\alpha)\gamma+\sqrt{s}}.
\end{align*}
This ends the proof as this expression coincides with the expression in equation \eqref{expectedintensity} for $\beta = 1/2$.
\end{proof}

\subsection{Inversion of the Laplace transform and limiting behaviour}

\textcolor{black}{For $\beta \in (0,1)$ and $\beta \neq 1/2$, the analytical inverse to equation \eqref{expectedintensity} can be found using a result from \citep{prabhakar1971} (formula 2.5). We report it below for readers' convenience. We define the Prabhakar functions a.k.a. three-parameter Mittag-Leffler functions as follows
\begin{equation}
\label{prabhakar}
E_{a,b}^c (z) := \frac{1}{\Gamma(c)} \sum_{k=0}^\infty \frac{\Gamma(c+k) z^k}{k! \Gamma(a k + b)}, \; a,b,c \in \mathbb{C}, \; \mathrm{Re}(a) > 0.
\end{equation}
The Laplace transform we need is
\begin{equation}
\label{prabhakarlt}
\int_0^\infty \mathrm{e}^{-st} t^{b-1} E_{a,b}^c (\nu t^b) \, dt = s^{-b} (1 - \nu s^{-a})^{-c}, \; \mathrm{Re}(a), \mathrm{Re}(b) > 0, |s| > |\nu|^{1/\mathrm{Re}(a)}.
\end{equation}
We can now state the following proposition
\begin{proposition}
The inverse Laplace transform of \eqref{expectedintensity} is
\begin{equation}
\label{analyticalinverse}
\lambda(t) = \frac{\Lambda_0}{1-\alpha} - \frac{\alpha \Lambda_0}{1-\alpha} E_{\beta} ((\alpha -1) \gamma t^\beta),
\end{equation}
where $E_\beta(z) := E_{\beta,1}^1 (z)$ is the one-parameter Mittag-Leffler function.
\label{propositionnew}
\end{proposition}
\begin{proof}
To prove the result, we first rewrite \eqref{expectedintensity} as
$$
\tilde{\lambda}(s) = \Lambda_0 \left[\frac{\gamma s^{-1}}{(1-\alpha) \gamma + s^\beta} + \frac{s^{\beta-1}}{(1-\alpha) \gamma + s^\beta} \right],
$$
then we apply \eqref{prabhakarlt} to the two summands to get
$$
\lambda(t) = \Lambda_0 [ \gamma t^\beta E_{\beta, \beta +1}^1 ((\alpha-1) \gamma t^\beta) + E_{\beta,1}^1 ((\alpha-1) t^\beta)]. 
$$
Now, it is possible to use the following identity for the two-parameter Mittag-Leffler function (see \citep{mathai08}, equation (2.2.1))
$$
E_{a,b}^1 (z) = z E_{a,a+b}^1 (z) + \frac{1}{\Gamma(b)},
$$
with $a = \beta$, $b=1$ and $z=(\alpha -1) \gamma t^\beta$ to get the thesis \eqref{analyticalinverse}.
\end{proof}
\begin{remark}
Proposition \ref{proposition1} can now be seen as a corollary of Proposition \ref{propositionnew} setting $\beta = 1/2$ and noticing that $E_{1/2,1}^1 (z) = \mathrm{e}^{z^2} \erfc (-z)$ (see \citep{mathai08}, equation (2.1.5)).
\end{remark}
To test the analytical results}, we have used an algorithm for the numerical inversion of Laplace transforms \citep{evans2000}. We have modified the original R code in the package {\em invLT} \citep{invLT} for our purposes and the code is available at {\tt https://github.com/habyarimanacassien/Fractional-Hawkes}. In Figure \ref{Fig:1}, we present a comparison between the numerical inversion of the Laplace transform and the analytical result for the cases $\beta = 1/2$ and $\beta = 0.9$; the values of the other parameters are specified in the figure caption.\begin{figure}
\includegraphics[width= \textwidth]{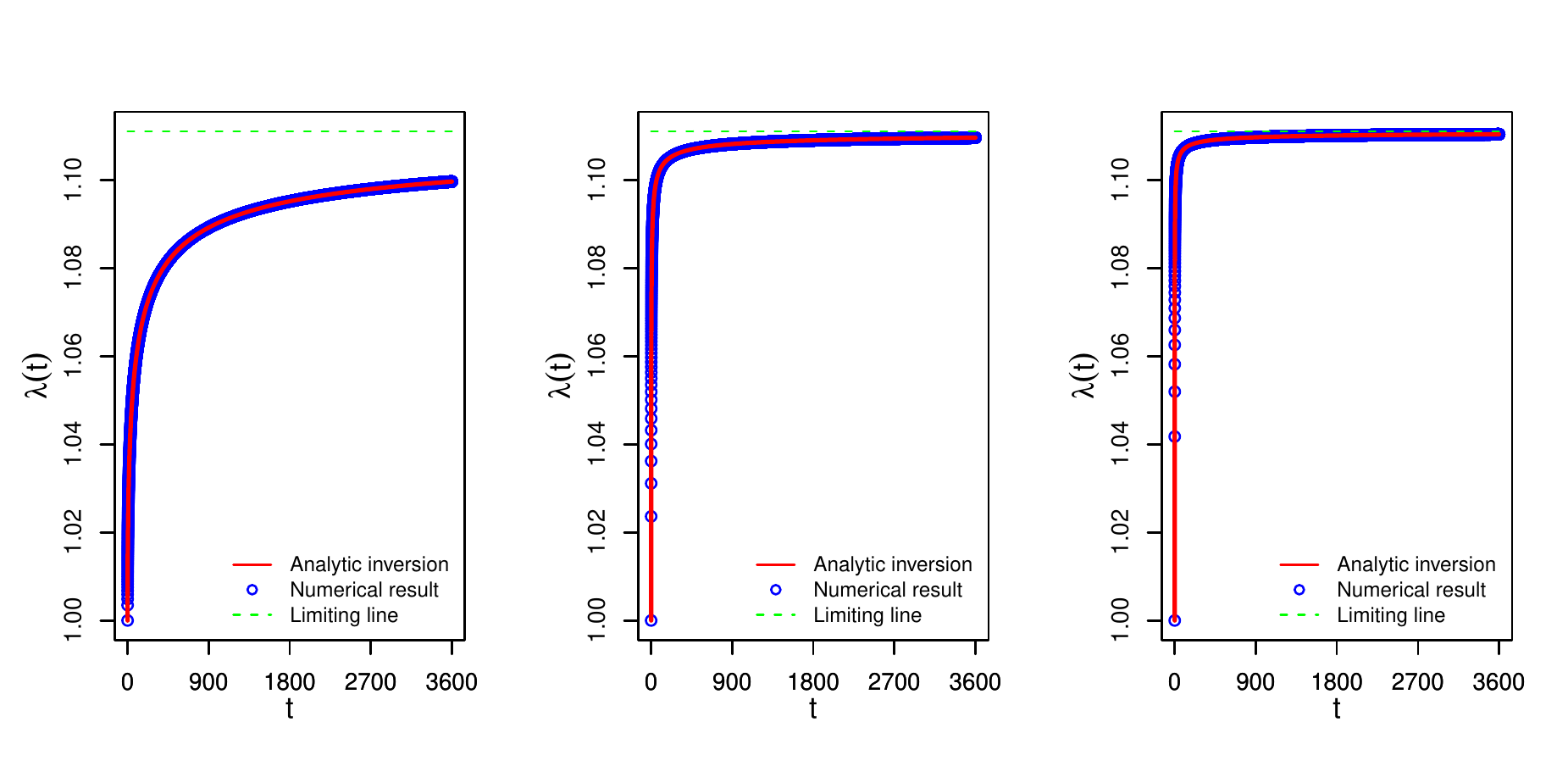}
\includegraphics[width= \textwidth]{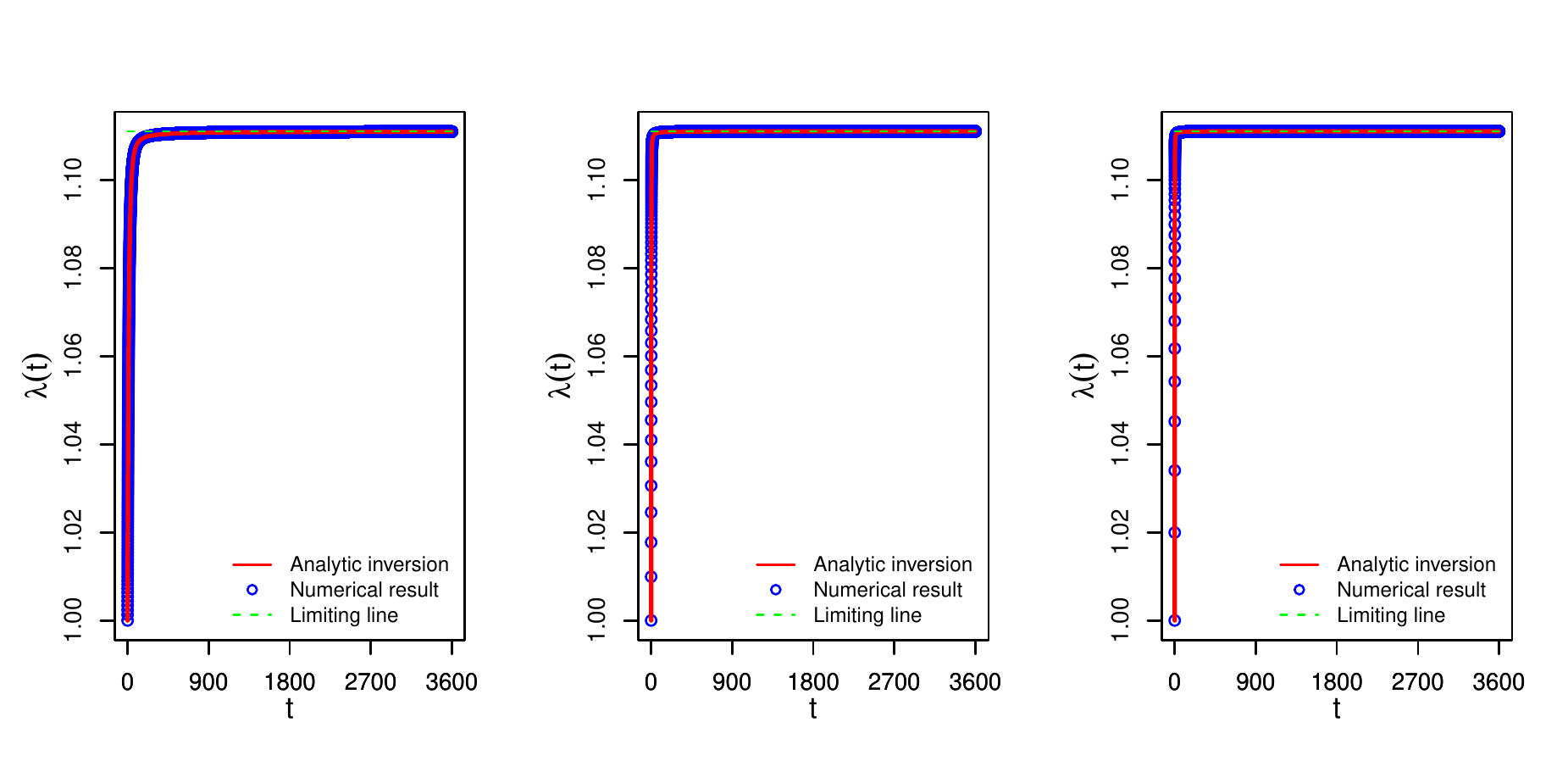}
\caption{(Colour online) Top: Comparison between the exact formula \eqref{inv1} for $\beta = 1/2$ (solid red line) and the numerical inversion of the Laplace transform in \eqref{expectedintensity} (blue dots). The parameters are
$\Lambda_0 = 1$, $\alpha = 0.1$ and $\gamma = 0.1, 0.8, 1.7$ from left to right. Bottom: Comparison between the exact formula \eqref{analyticalinverse} for $\beta = 0.9$ (solid red line) and the numerical inversion of the Laplace transform in \eqref{expectedintensity} (blue dots). The parameters are $\Lambda_0 = 1$, $\alpha = 0.1$ and $\gamma = 0.1, 0.8, 1.7$ from left to right. Time is measured in arbitrary units.}
\label{Fig:1}
\end{figure}
%\begin{figure}
%\includegraphics[width= \textwidth]{Figure1new.pdf}
%\caption{(Colour online) Comparison between the exact formula \eqref{analyticalinverse} for $\beta = 0.9$ (solid red line) and the numerical inversion of the Laplace transform in \eqref{expectedintensity} (blue dots). The parameters are $\Lambda_0 = 1$, $\alpha = 0.1$ and $\gamma = 0.1, 0.8, 1.7$ from left to right. Time is measured in arbitrary units.}
%\label{Fig:1new}
%\end{figure}

As mentioned in \citep{chen20}, one can observe a fast increase of the expected intensity, followed by a slower convergence to a constant value. In that paper, we mentioned that one can prove that the following limit holds
\begin{equation}
\label{limitingbehaviour}
\lim_{t \to \infty} \lambda(t) = \frac{\Lambda_0}{1 - \alpha}
\end{equation}
using Tauberian theorems. We now make this result explicit. We first quote, without proof and as a lemma, the Tauberian theorem proven by Hardy and Littlewood \citep{hardy1930}. 
\begin{lemma}[Hardy and Littlewood 1930, theorem 1]
\label{tauberian1}
If $f(t)$ is positive and integrable over every finite range $(0,T)$ and $\mathrm{e}^{-st} f(t)$ is integrable over $(0, \infty)$ for every $s >0$ and if
\begin{equation}
\label{asymptotic1}
\tilde{f} (s) = \int_0^\infty \mathrm{e}^{-st} f(t) \, dt  \sim H s^{-a}
\end{equation}
for $s \to 0$ with $a >0$, $H>0$, then as $t \to \infty$
\begin{equation}
\label{asymptotic2}
F(t) = \int_0^t f(u) \, du \sim \frac{H}{\Gamma(1+a)} t^a.
\end{equation}
\end{lemma}
In our case, we also have the following lemma
\begin{lemma}
\label{asymptotic5}
Consider $\tilde{\lambda} (s)$ given by equation \eqref{expectedintensity}. Then we have
\begin{equation}
\label{asymptotyc3}
\tilde{\lambda} (s) \sim \frac{\Lambda_0}{1 - \alpha} s^{-1}.
\end{equation}
\end{lemma}
\begin{proof}
For $s \to 0$, equation \eqref{expectedintensity} can be expanded around $s =0$ in Puiseux-Newton series \citep{puiseux} as
\begin{equation}
\label{puiseux}
\tilde{\lambda} (s) = \frac{\Lambda_0}{1-\alpha} s^{-1} - \sum^{\infty}_{k=1}\frac{\alpha\Lambda_0}{(\alpha-1)^{k+1}\gamma^{k}}s^{k\beta-1}
\end{equation}
meaning that
\begin{equation}
\lim_{s \to 0} \frac{\tilde{\lambda}(s)}{\frac{\Lambda_0}{1-\alpha}s^{-1}} = 1.
\end{equation}
\end{proof}
An immediate consequence of the two previous lemmata is the following proposition
\begin{proposition}
\label{asymptotic4}
Let $\lambda(t) = \mathcal{L}^{-1} \{ \tilde{\lambda} (s) \} (t)$ with $ \tilde{\lambda} (s)$ given by equation \eqref{expectedintensity}. Then $$\lim_{t \to \infty} \lambda(t) = \frac{\Lambda_0}{1-\alpha}.$$
\end{proposition}
\begin{proof}
With the identification $H = \Lambda_0/(1 - \alpha)$ and
$a = 1$, we have that the expected number of events asymptotically grows in a linear way by virtue of Lemma 1 and Lemma 2 for large $t$
\begin{equation}
\label{expnasympt}
\mathbb{E} [N(t)] = \int_0^t \lambda (u) \,du \sim \frac{\Lambda_0}{1-\alpha} t.
\end{equation}
Taking the derivative of equation \eqref{expnasympt} leads to the thesis
\begin{equation}
\label{lambdasympt}
\lambda (t) \sim \frac{\Lambda_0}{1-\alpha}.
\end{equation}
\end{proof}

\section{Expected number of events}
\label{sec:ENE}
As we mentioned previously, the expected number of events, $\mathbb{E}[N(t)]$ up to time $t$ is given by the integral of the expected intensity from the initial time $0$ to $t$. From Proposition 2, we can see that for large $t$, it grows linearly in $t$ with slope given by $\Lambda_0/(1-\alpha)$. In the case $\beta = 1/2$, we can directly integrate equation \eqref{inv1} and compare the result with Monte Carlo simulations of the process using the program described in \citep{chen20} and available at {\tt https://github.com/habyarimanacassien/Fractional-Hawkes}. This gives the following exact analytical result
\begin{multline}
\label{expn1}
\mathbb{E}[N(t)] = \int_0^t \lambda(u) \, du = \frac{\Lambda_0}{1-\alpha}t-\frac{\alpha\Lambda_0}{1-\alpha}\int^t_0e^{(1-\alpha)^2\gamma^2u}\erfc{\left((1-\alpha)\gamma\sqrt{u}\right)}\,du \\
=\frac{\Lambda_0}{1-\alpha}t-\frac{\alpha\Lambda_0}{(1-\alpha)^3\gamma^2}\int^{(1-\alpha)^2\gamma^2t}_0e^{w}\erfc{(\sqrt{w})}\,dw
\end{multline}
where $w = (1-\alpha)^2 \gamma^2 u$. The integral in the right hand side of equation (\ref{expn1}) can be computed and gives the following result
\begin{multline}
\mathbb{E}[N(t)]=\frac{\Lambda_0}{1-\alpha}t \\
-\frac{\alpha\Lambda_0}{(1-\alpha)^3\gamma^2\sqrt{\pi}}\left(\sqrt{\pi}e^{(1-\alpha)^2\gamma^2t}\erfc{((1-\alpha)\gamma\sqrt{t})}+2(1-\alpha)\gamma\sqrt{t}-\sqrt{\pi}\right). \label{expn2}
\end{multline}
\begin{remark}
From equation \eqref{expn2}, one can see that
$$
\lim_{t \to \infty} \frac{\mathbb{E}[N(t)]}{t} = \frac{\Lambda_0}{1 - \alpha},
$$
or, in other words, $\mathbb{E}[N(t)] \sim \Lambda_0 t /(1 - \alpha)$ in agreement with equation \eqref{expnasympt} in Proposition 2.
\end{remark}
The comparison between the analytical result \eqref{expn2} and Monte Carlo simulations is presented in Figure \ref{Fig:2} for values of the parameters given in the caption.
\begin{figure}
\includegraphics[width= \textwidth]{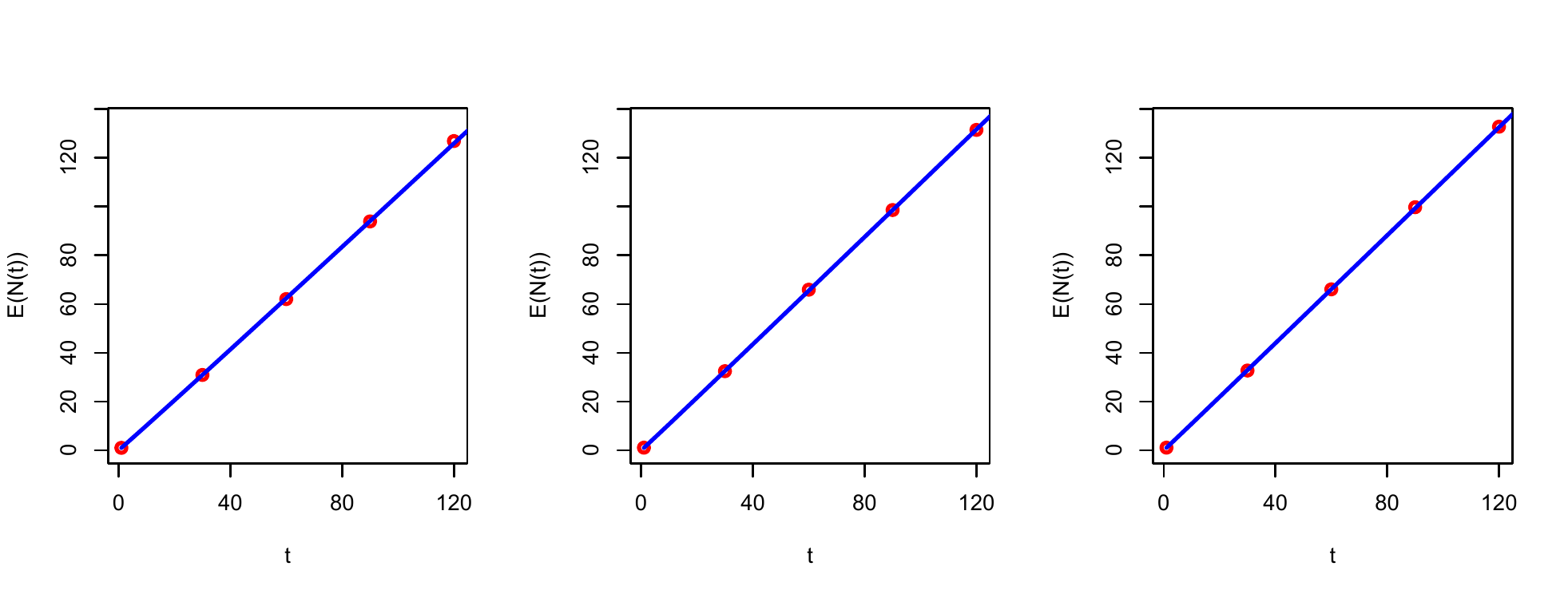}
\caption{(Colour online) Comparison between the exact formula \eqref{expn1} for $\beta = 1/2$ (solid blue line) and Monte Carlo simulations (red circles). The parameters are $\Lambda_0 =1$, $\alpha=0.1$ and $\gamma = 0.1,0.8,1.7$ from left to right. Time is measured in arbitrary units.}
\label{Fig:2}
\end{figure}

In the general case $\beta \in (0,1)$ and $\beta \neq 1/2$, one can first numerically invert the Laplace transform using the procedure outlined in Section 2.2 and then numerically integrate the inverse to obtain an estimate of 
the expected number of events up to time $t$. The program to do so is available at the previously mentioned repository: {\tt https://github.com/habyarimanacassien/Fractional-Hawkes}. Results for the case $\beta =0.99$ are presented in Figure \ref{Fig:3} for values of the other parameters given in the caption. This curve is compared with Monte Carlo simulations of the process. \textcolor{black}{It is also possible to integrate \eqref{analyticalinverse} with respect to $t$ to get
\begin{equation}
\mathbb{E}[N(t)] = \frac{\Lambda_0}{1-\alpha} t - \frac{\alpha \Lambda_0}{1-\alpha} t E_{\beta,2}^1 ((\alpha-1) \gamma t^\beta),
\label{exactexpectation}
\end{equation} 
where
$$
t E_{\beta,2}^1 ((\alpha -1) \gamma t^\beta) = \int_0^t E_\beta ((\alpha -1) \gamma u^\beta) \, du.
$$
This integral is an immediate consequence of equation (2.3.17) in \cite{mathai08} 
$$
\int_0^z t^{c-1} E_{b,c} (w t^b) \, dt = z^c E_{b,c+1}^1 (w z^\beta),
$$
when we set $b = \beta$, $c=1$ and $w = (\alpha-1) \gamma$. The same identity can be used to see that equation \eqref{exactexpectation} reduces to \eqref{expn1} (or equivalently \eqref{expn2}) for $\beta = 1/2$.}
\begin{figure}
\includegraphics[width= \textwidth]{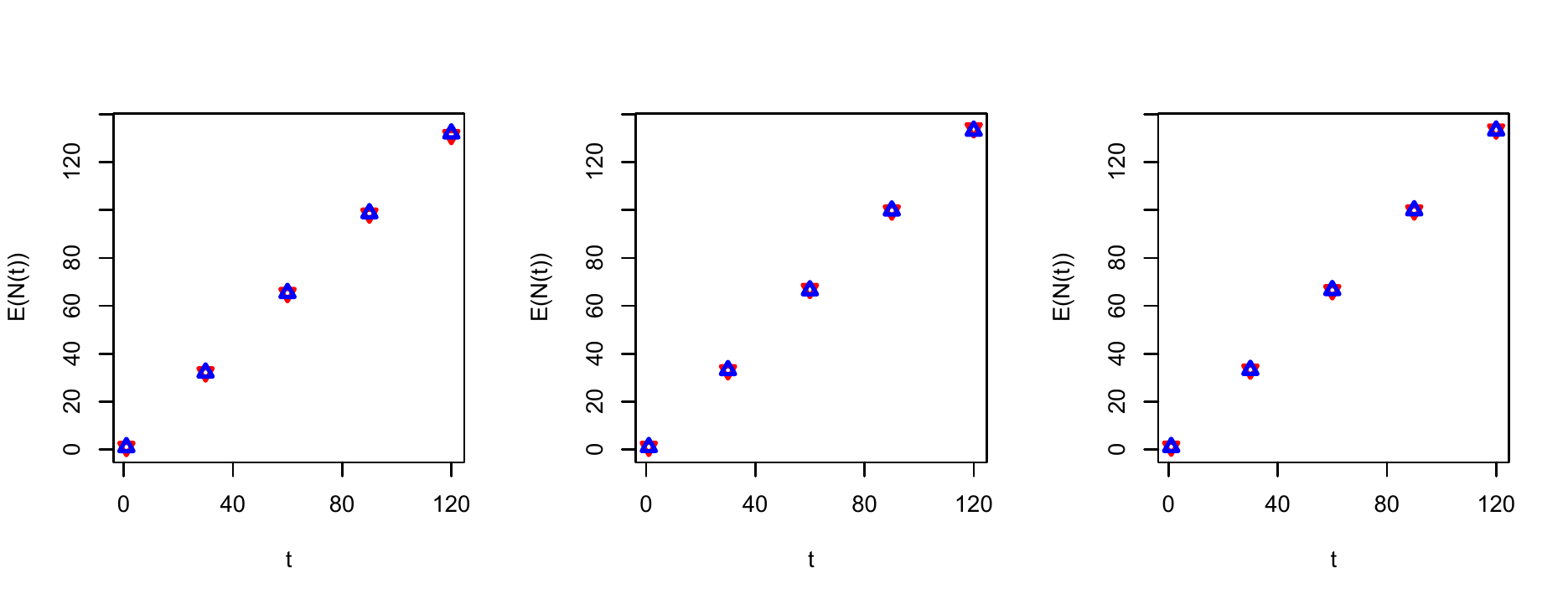}
\caption{(Colour online) Comparison between the integral of the numerical inversion of the Laplace transform (blue triangles) and Monte Carlo simulations (red triangles) for $\beta = 0.99$. The parameters are $\Lambda_0 =1$, $\alpha=0.1$ and $\gamma = 0.1,0.8,1.7$ from left to right. Time is measured in arbitrary units.}
\label{Fig:3}
\end{figure}

\section{Distribution of the number of events}
\label{sec:dist}

In the previous section, the expected number of events up to time $t$ is discussed. In principle, it is possible to derive the distribution of the random variable $N(t)$ at any fixed time using the Monte Carlo simulation of the process. The program for this task is available at the GitHub repository {\tt https://github.com/habyarimanacassien/Fractional-Hawkes}. Unfortunately, for the distribution of this random variable, we do not have any exact formula, but it is possible to compare it with the Poisson distribution in the limit of ``small'' branching ratio $\alpha$. This is done in Figure \ref{Fig:4}, for $\Lambda_0=1$ and $\gamma=1$, when $\alpha=0.01$, $t=1, 5, 10$ from top to bottom and $\beta=0.5, 0.9$ from left to right. We can see that, with this branching ratio, the behaviour of the counting process is similar to the behaviour of the Poisson process with the same values of $\Lambda_0$ and $\gamma$ as expected.
\begin{figure}
\includegraphics[width= \textwidth]{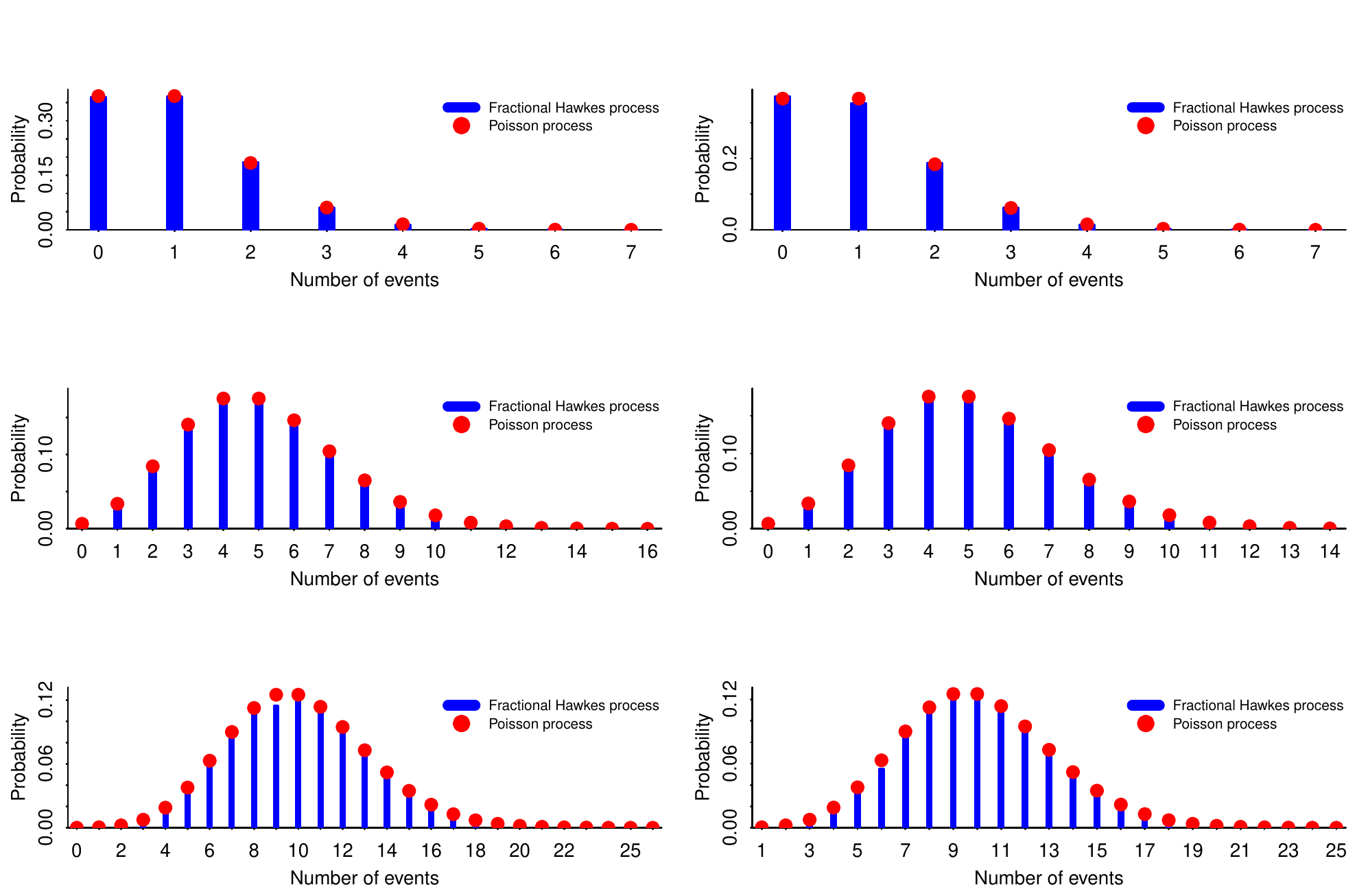}
\caption{(Colour online) Comparison between the distribution $\mathbb{P}(N(t) = k)$ for the fractional Hawkes process (vertical blue bars) and for the corresponding limiting Poisson process (red dots). The parameters are $\Lambda_0=1$, $\gamma=1$, $\alpha=0.01$, $t=1, 5, 10$ from top to bottom and $\beta=0.5, 0.9$ from left to right. Time is measured in arbitrary units.}
\label{Fig:4}
\end{figure}
There is another important limit that can be used as a test for the Monte Carlo program. When $\beta = 1$, the Mittag-Leffler kernel coincides with the exponential kernel and one can expect that, for $\beta$ close to $1$, the behaviour of the fractional Hawkes process is close to the one of the standard Hawkes process with an exponential kernel. This limit is illustrated in Figure \ref{Fig:5} when $\beta=0.99$, $t=1, 5, 10$ from top to bottom and $\alpha=0.1, 0.5$ from left to right. Also in this case, we set $\Lambda_0=1$ and $\gamma=1$. Also in this limit, one can see that the two processes behave in a similar way.
\begin{figure}
\includegraphics[width= \textwidth]{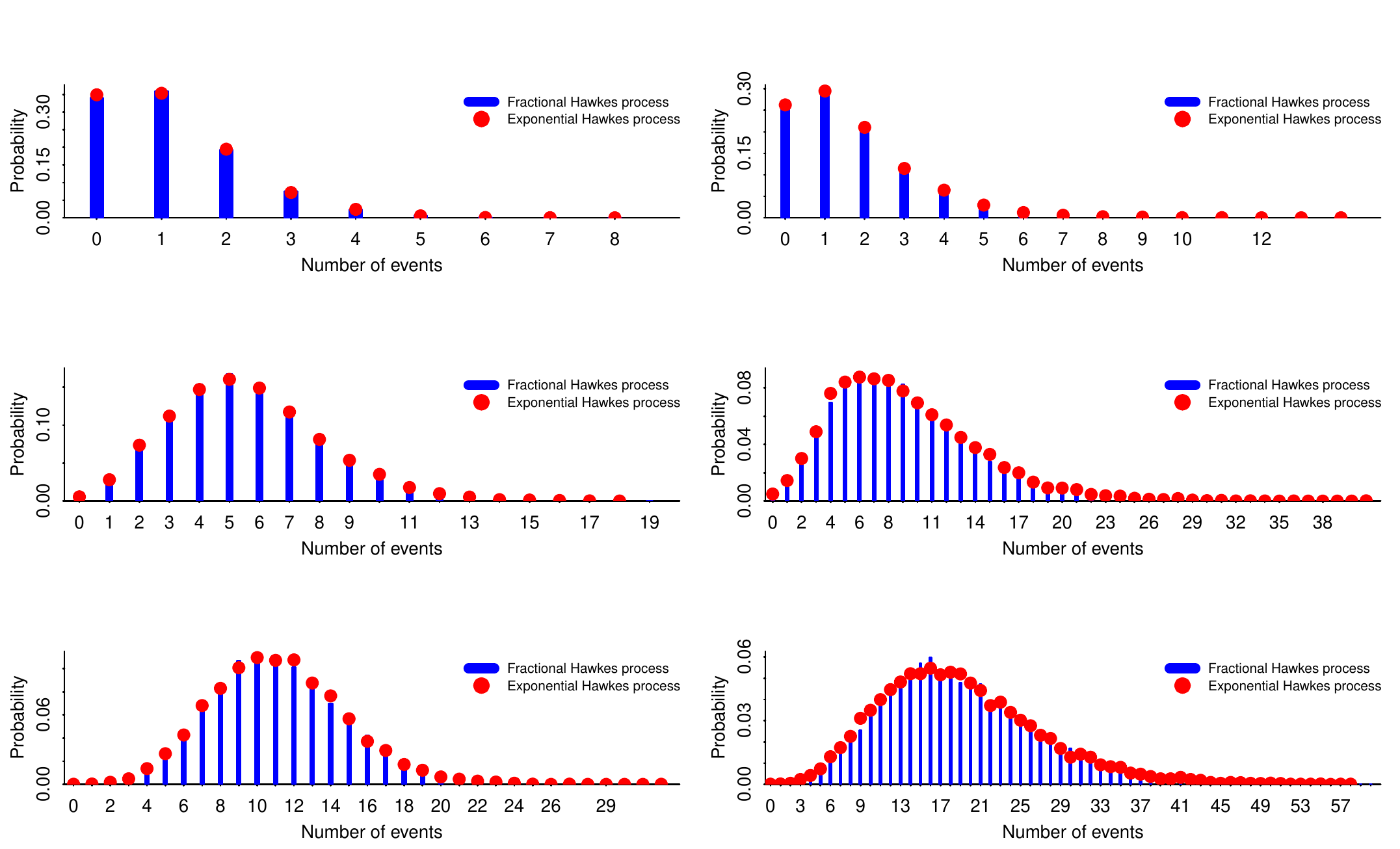}
\caption{(Colour online) Comparison between the distribution $\mathbb{P}(N(t) = k)$ for the fractional Hawkes process (vertical blue bars) and for limiting Hawkes process with exponential kernel (red dots). The parameters are $\Lambda_0=1$, $\gamma=1$, $\beta=0.99$, $t=1, 5, 10$ from top to bottom and $\alpha=0.1, 0.5$ from left to right. Time is measured in arbitrary units.}
\label{Fig:5}
\end{figure}
To conclude this section, we consider the case in which $\alpha = 0.5$ and we set $\beta = 0.5, 0.9$ as in Figure \ref{Fig:4}. In this case, represented in Figure \ref{Fig:6}, one can see that the Poisson approximation does no longer work as the effect of the self-exciting part is now stronger.
\begin{figure}
\includegraphics[width= \textwidth]{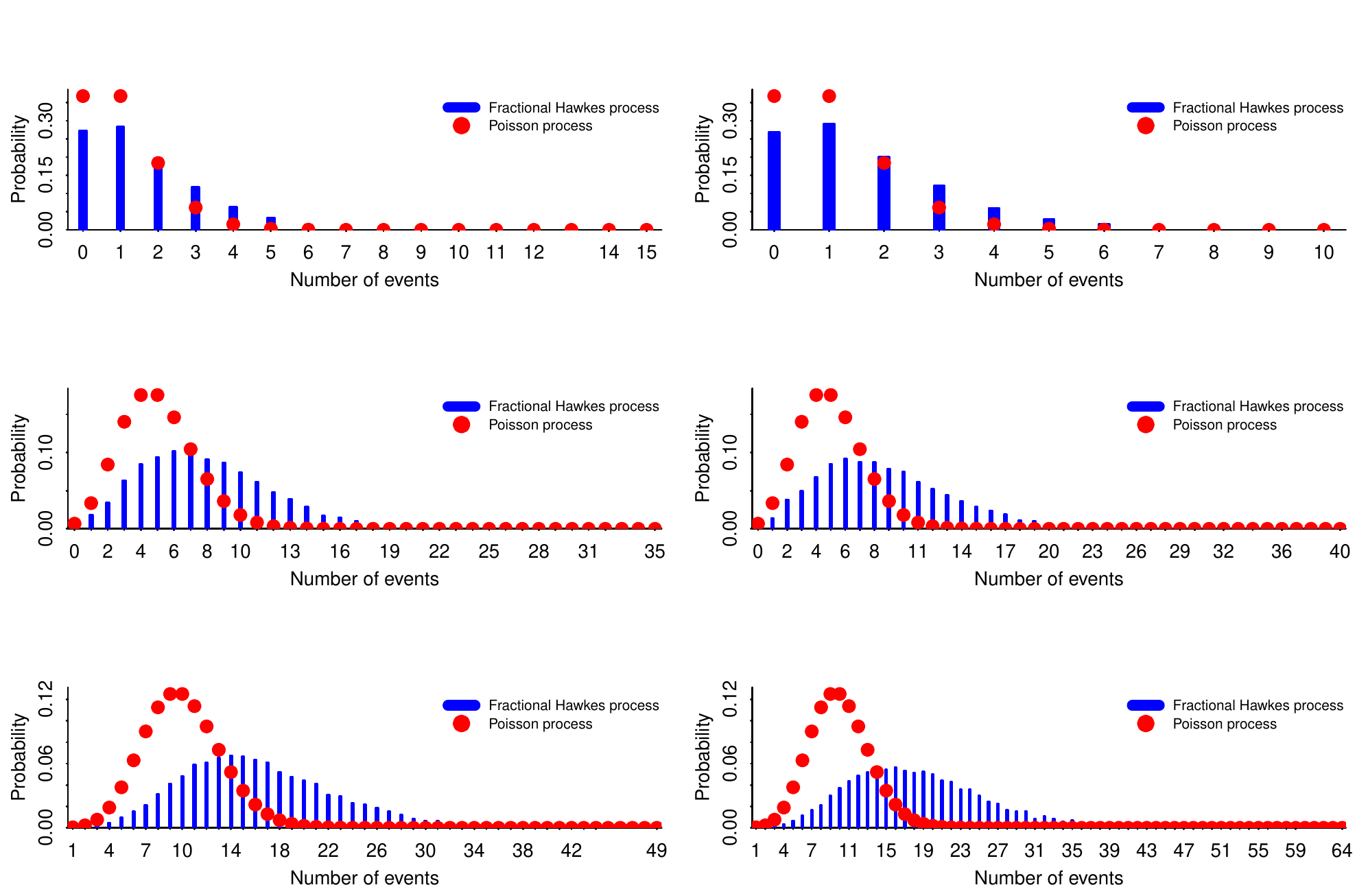}
\caption{(Colour online) Comparison between the distribution $\mathbb{P}(N(t) = k)$ for the fractional Hawkes process (vertical blue bars) and for the corresponding limiting Poisson process (red dots). The parameters are $\Lambda_0=1$, $\gamma=1$, $\alpha=0.5$, $t=1, 5, 10$ from top to bottom and $\beta=0.5, 0.9$ from left to right. Time is measured in arbitrary units.}
\label{Fig:6}
\end{figure}

\section{Conclusions and Outlook} 

In \cite{chen20}, we introduced a Hawkes process of fractional type by using a kernel proportional to the probability density function of Mittag-Leffler random variables. In this paper, we characterize the process by including explicit proofs of some results that were announced in \cite{chen20} and we further present simulations to illustrate and corroborate our analytical results. 

Our next step will be to study the efficiency and performance of procedures estimating parameters of the process. In particular, we are currently working on approximate Bayesian computation (ABC) algorithms \cite{abc} and we plan this to be the subject of a forthcoming paper. We also plan work on multivariate Hawkes processes which will extend the applicability of the method.

\section*{Acknowledgments}

\noindent Enrico Scalas would like to thank the Isaac Newton Institute for Mathematical Sciences, Cambridge, for support and hospitality during the programme Fractional Differential Equations (FDE2) where work on this paper was undertaken. This work was supported by EPSRC grant no EP/R014604/1. Enrico Scalas was also partially funded by the Dr Perry James (Jim) Browne Research Center at the Department of Mathematics, University of Sussex. \textcolor{black}{Federico Polito was partially supported by INdAM/GNAMPA. Finally, insightful discussions with Donatien Hainaut and Nicos Georgiou are acknowledged.}

\end{document}